%% file: main.tex
\documentclass{amsart}

\usepackage{amsmath,amssymb}
\usepackage{mathrsfs}
\usepackage{comment}
\usepackage{tikz-cd}
\usepackage{graphicx}

\usepackage{color} 
\usepackage[linktocpage=true]{hyperref}
\hypersetup{
    colorlinks=true, 
    linkcolor=blue,  
}

\numberwithin{equation}{section}

\theoremstyle{plain}
\newtheorem{theorem}{Theorem}[section]
\newtheorem{proposition}[theorem]{Proposition}
\newtheorem{lemma}[theorem]{Lemma}

\theoremstyle{definition}

\newtheorem{definition}{Definition}[section]

\DeclareMathOperator{\Tr}{Tr}

\def\R{\mathbb{R}}
\def\Z{\mathbb{Z}}

\def\C{\mathbb{C}}

\def\S{\mathbb{S}}
\def\D{\partial}
\def\Dd{\mathcal{D}}

\DeclareMathOperator{\diff}{Diff}

\def\tnabla{\widetilde{\nabla}}
\def\tnablax{\tnabla_{\partial_x}}
\def\bnabla{{}^b\! \nabla}
\def\beb{{}^b \overline{\eta}}
\def\be{{}^b \eta}
\def\b{{}^b\! }

\DeclareMathOperator{\bTr}{{}^{\mathit{b}} \! Tr}

\newcommand{\bint}{{}^{b}\hspace{-1ex}\int}

\def\dbar{d\hspace*{-0.08em}\bar{}\hspace*{0.1em}}

\def\ab{{}_{ab}}
\def\absym{^{ab}\!\sigma}
\DeclareMathOperator{\tf}{\textbf{tf}}
\DeclareMathOperator{\ff}{\textbf{bf}}

\def\Dt{\Dd_t}
\def\DD{(\Dt^2)^{-\frac{s+1}{2}}}
\def\DDD{\Dt(\Dt^2)^{-\frac{s+1}{2}}}

\def\t{\tau}

\begin{document}

\title[The ab-calculus and the global anomaly formula]
{The Adiabatic b-Calculus and the Global Anomaly Formula on Manifolds with Boundary}

\author{Chad Nelson}
\address{Department of Mathematics \\Binghamton University\\
Vestal Parkway East\\Binghamton\\NY 13902\\ U.S.A. }
\email{cnelson5@binghamton.edu}

\begin{abstract}
We study the determinant line bundle for a family of elliptic operators on compact manifolds with boundary with the associated Quillen metric and Bismut-Freed connection, and we prove the global anomaly formula in this context. This extends Witten's holonomy theorem to manifolds with boundary. Witten studied
the limiting behavior of the eta invariant of a Dirac operator on the total space of a fibration over $\mathbb{S}^1$ when the metric in the base $\mathbb{S}^1$ is
blown up in a certain manner, a process known as taking the
adiabatic limit. In this paper, we construct the calculus of adiabatic $b$-pseudodifferential operators, which naturally synthesizes the $b$-calculus and the adiabatic calculus of pseudodifferential operators. Using this calculus, and assuming the fiber operators to be invertible, we compute the
adiabatic limit of the natural eta invariant in this setting and relate it to the integral of the Bismut-Freed connection and new correction terms coming from the boundary of the fibers.
\end{abstract}


\maketitle

{\textbf{Acknowledgments.} This paper is based on my PhD dissertation. I would like to thank my doctoral advisor, Professor Paul Loya, for introducing me to the world of pseudodifferential operators and singular analysis, as well as for proposing this project. I also wish to thank him for all of the guidance and encouragement he has given me over the years.   }

\tableofcontents

\input{intro}

\input{abcalc}
\input{witten}

\bibliographystyle{amsplain}
\def\cprime{$'$} \def\polhk#1{\setbox0=\hbox{#1}{\ooalign{\hidewidth
  \lower1.5ex\hbox{`}\hidewidth\crcr\unhbox0}}}
\providecommand{\bysame}{\leavevmode\hbox to3em{\hrulefill}\thinspace}
\providecommand{\MR}{\relax\ifhmode\unskip\space\fi MR }
\providecommand{\MRhref}[2]{%
  \href{http://www.ams.org/mathscinet-getitem?mr=#1}{#2}
}
\providecommand{\href}[2]{#2}

\end{document}

%% file: intro.tex
\section{Introduction}
Determinant line bundles were introduced into the world of mathematics and physics in \cite{QuiD85}. Quillen considers a smooth complex vector bundle $E$ over a compact one-dimensional complex manifold and the space $\mathscr{A}$ of index zero Cauchy-Riemann operators on $E$. Quillen solved the following problem: determine, in a canonical way, a determinant function $\mathrm{DET}: \mathscr{A} \to \C$ such that $\mathrm{DET}\, \overline{\D} \ne 0$ if and only if $\overline{\D}$ is invertible. His solution was to assign to each $\overline{\D} \in \mathscr{A}$ its determinant line $\mathcal{L}_{\overline{\D}}$, which all together define a holomorphic line bundle $\mathcal{L} \to \mathscr{A}$, the determinant line bundle. The key to Quillen's construction is finding an everywhere parallel trivialization. In general, the obstruction to finding a parallel trivialization is called the \emph{anomaly} and both the curvature (the local anomaly) and the holonomy (the global anomaly) are exactly the obstructions to its existence. For applications of determinant lines to mathematics and physics, see, for example, \cite{AtiM-81-84-2597,IBisJ93,ICheJ87,BDoS-KrP90,FrD86,IFrD87,ISiI88,IWitE85,WiE85}.

In \cite{BiJ-FrDI86}, Bismut and Freed introduce Quillen's construction in the families setting. To describe their result, let $\phi :M \to B$ be a fiber bundle of closed manifolds with even-dimensional fibers having smoothly varying spin structures and Riemannian metrics. Let $E$ be a Hermitian vector bundle over $M$ and let $\Dd = \{\Dd_x\}_{x \in B}$ be the family of spin Dirac operators along the fibers; thus, the base $B$ plays the role of the parameter space $\mathscr{A}$ in Quillen's original construction. From this data, Bismut and Freed showed how to define a determinant line bundle $\mathcal{L} \to B$ with a canonical section. Given a unitary connection on $E$ and a connection (horizontal distribution) on $TM$ one can then define a metric (the Quillen metric) on $\mathcal{L}$ with a compatible connection (the Bismut-Freed connection). In \cite{BiJ-FrDII86} they compute the local anomaly using Bismut's heat equation approach to the families index \cite{BisJ85,QuiDII85} and they prove independently with Cheeger \cite{ChJ87} the global anomaly formula, verifying a conjecture of Witten \cite{WiE85} relating the global anomaly with the adiabatic limit of the eta invariant. To describe the global anomaly formula more precisely, let $B = \mathbb S^1$ with its trivial (non-bounding) spin structure and consider the \emph{adiabatic metric} on $M$ defined by
\[
g_t = \frac{\phi^* d x^2}{t^2} + g_{M/B}\ ,\ \ t \in (0,\infty)
\]
where $x$ denotes the standard angular coordinate on $\mathbb S^1$ and $g_{M/B}$ is the family of fiber metrics, and let $\Dd_t$ be the corresponding spin Dirac operator over $M$.  As $t \to 0$, we see that
the base blows up.  Since $M$ is
compact, for each $t > 0$ the Atiyah-Patodi-Singer \cite{AtM-PaV-SiI75} eta invariant $\eta(\Dd_t)$ is defined.
Assuming the family of Dirac operators on the fibers are invertible, Witten's holonomy theorem, or the global anomaly formula, is the following:
\[
\lim_{t \to 0} i \pi\, \eta(\Dd_t) = - \log
\,\mathrm{hol} \,(\mathcal{L}) ,
\]
where $\mathcal{L}$ denotes the determinant line bundle of the
family $\Dd^+$. In \cite{MorS04}, Moroianu extends this result to general elliptic families of first order differential operators, which is significant because previous proofs systematically employed elements special to Dirac operators such as anticommuting variables; moreover, he computes the adiabatic limit of the full eta function and not just the eta invariant at $s = 0$. For further generalizations and applications of the adiabatic limit of the eta invariant, see \cite{BiJ-ChJ89,BiJ-ChJ92,DaiX91,IDaiX93,DaX02,DaX-FrD94,DelSVII87,DelSVI87,DoR-WoK91,LotJ94,
NiL98,PaJ-WoK02III,PaJ-WoK02,PaJ-WoK06} to name a handful of references.

In \cite{PiaP96,PiaP98} Piazza extended the local anomaly formula to the case when the fibers are compact manifolds with boundary. The viewpoint taken by Piazza is that the analysis of operators on compact manifolds with boundary is more lucid and more directly proved when recast in the language of complete manifolds with (asymptotically) cylindrical ends. Here, one has available Melrose's $b$-geometry and $b$-calculus of pseudodifferential operators \cite{BMeR93} (one could also use the cusp calculus \cite{IMeR96}), which can be used not only to investigate the spectral properties of operators but also to correctly formulate the problem so that proofs are actually tractable \cite{MaR-PiP98,IPiP97}. The goal of this paper is to prove the global anomaly formula for manifolds with boundary using the $b$-machinery.

Let $\phi : M \to B$ be a fiber bundle where both $M$ and $B$ are compact and where the fibers are diffeomorphic to a compact manifold with boundary $Y$. We make the following assumptions on the fiber bundle $\phi : M \to B$:
\begin{enumerate}
\item Fix a boundary defining function $\rho$ and fix a corresponding family $g_{M/B}$ of \textit{exact} $b$-metrics on the fibers, where \textit{exact} means that
\[
g_{M/B} = \frac{d\rho^2}{\rho^2} + h,
\]
with $h \in C^\infty(M,T^*(M/B)\otimes T^*(M/B))$.

\item Fix a \emph{$b$-connection}, which is a horizontal distribution on the fiber bundle $\phi : M \to B$ compatible with the boundary defining function $\rho$.

\item Finally, we fix a Hermitian $\Z_2$-graded vector bundle $E = E^+ \oplus E^- \to M$ and a unitary $\Z_2$-graded connection $\nabla$ on $E$.

\end{enumerate}

The above assumptions induce a connection $\widetilde{\nabla}$ on the infinite dimensional bundle over $B$ whose fiber over a point $x \in B$ is the space $\dot{C}^\infty(M_x,E|_{M_x})$, where $M_x$ is the fiber above $x$, consisting of sections of $E|_{M_x}$ vanishing to infinite order at $\D M_x$.  The $b$-determinant line bundle and the $b$-Quillen metric for a family of Fredholm first order $b$-pseudodifferential operators acting along the fibers are constructed by Piazza in \cite{PiaP96}.

For the rest of the paper we restrict ourselves to families $\Dd^+ = \{\Dd^+_x\}_{x \in B}$ of first order elliptic $b$-differential operators,
\begin{equation*}
\Dd^+_x : C^\infty (M_x , E^+|_{M_x}) \to C^\infty(M_x, E^-|_{M_x}),
\end{equation*}
where $M_x$ is the fiber above $x \in B$. We assume the operators are invertible at each $x\in B$. We assume that this family is of Dirac-type at the boundary, which essentially means that at the boundary we can write
\begin{equation*} 
\Dd^+ \sim G \big(\rho \D_{\rho} + \Dd \big),
\end{equation*}
where $G: E^+|_{\partial M} \to E^-|_{\partial M}$ is a unitary isomorphism and $\Dd = \{\Dd_{x}\}_{x \in B}$ is a family of invertible, self-adjoint operators on $\partial M$ and $\sim$ is made precise using the notion of indicial operators, see Section \ref{sec_witten}. Since $\Dd$ is a family of Fredholm operators we can define the total normalized eta function at the boundary:
\begin{equation} \label{bebintro}
\overline{\eta}(s, \D M) : = \frac{\Gamma\left( \frac{s + 1}{2}
\right)}{\sqrt{\pi}}  \Tr \Big( (\Dd^2)^{-
\frac{s + 1}{2}} \Dd \Big),
\end{equation}
where $\Tr$ denotes the fiber-wise trace. The function $\overline{\eta}(s,\D M)$ is smooth on the base $B$ and is meromorphic as a function of $s \in \C$.

In Section \ref{sec_abcalc} we introduce the adiabatic $b$-calculus, or the $ab$-calculus, and we use this calculus to prove the global anomaly formula in Section \ref{sec_witten}. The $ab$-calculus is a natural synthesis of the small $b$-calculus $\Psi_b$ described by Melrose \cite{BMeR93} and the adiabatic calculus $\Psi_a$ described by Moroianu \cite{MorS04, HomologyAndAdiabatic}. Our proof of the global anomaly theorem closely imitates that of \cite{MorS04}, with the key difference being the presence of the $b$-trace functional. This is an extension of the standard trace functional for elements in $\Psi^{-\infty}_b(M)$, despite these elements not being trace class in general.

In section \ref{sec_witten} we are interested in the holonomy of the $b$-determinant line bundle $\mathcal{L}$ along the base. Thus, we shall restrict the fiber bundle $\phi : M \to B$ to a curve on the base $B$, which we take to be diffeomorphic to $\S^1$.

Let $\D_x$ denote the canonical positively oriented unit vector field on $\S^1$. For $t > 0$ consider the self-adjoint operator
\[
\Dd_t = \begin{pmatrix} - t i \widetilde{\nabla}_{\D_x} &
\Dd^- \\
\Dd^+ & t i \widetilde{\nabla}_{\D_x} \end{pmatrix} : C^\infty(M,E) \to
C^\infty(M,E)
\]
where $\Dd^- = (\Dd^+)^*$, the adjoint with respect to the metric $g_{M/B}$ on the fibers, and  $\widetilde{\nabla}$ is the unitary connection on the
infinite-dimensional bundle of fiber-wise sections vanishing to
infinite order at the boundary.

The operator $\Dd_t$ is an elliptic $ab$-differential operator on $M$ that is invertible for $t > 0$ sufficiently small (Proposition \ref{prop_invert}). In particular, for $t > 0$ we can define its normalized $b$-eta function
\[
\beb (s,\Dd_t)  = \frac{\Gamma\left( \frac{s + 1}{2}
\right)}{\sqrt{\pi}} \bTr ( ( \Dd_t^2)^{- \frac{s +
1}{2}} \Dd_t )  .
\]
It is well-known \cite{LauR-MorS02,LoP01,LoPII01} that $\bTr ( ( \Dd_t^2)^{- \frac{s +
1}{2}} \Dd_t )$ defines a meromorphic function of $s \in \C$ with at most simple poles at the arithmetic sequence $n-1,n-3,n-5,\ldots$ where $n = \dim M$, and $s = 0$ is a regular point \cite{GilP81}.

The $b$-Bismut-Freed meromorphic family of one-forms on $B$ is by definition the meromorphic
family of one-forms
\[
\b\omega(s) = \frac{\Gamma \left( \frac{s}{2} + 1 \right)}{2}
\bigg( \bTr \left( (\Dd^- \Dd^+)^{- \frac{s}{2} - 1}
\Dd^- \dot{\Dd}^+  \right)
+ \bTr \left( \Dd^- \dot{\Dd}^+   (\Dd^- \Dd^+)^{- \frac{s}{2} -
1} \right) \bigg) ,
\]
where $\bTr$ denotes the fiber-wise $b$-trace and $\dot{\Dd}^+:= [\widetilde{\nabla} , \Dd^+]$ denotes differentiation of $\Dd^+$ by $\widetilde{\nabla}$.

We do \emph{not} use Quillen's super-formalism \cite{QuiDII85} in this paper. Thus, we define
\begin{equation*} \label{nonsuper}
\Dd^- \dot{\Dd}^+ : = dx \otimes \Dd^- \widetilde{\nabla}_{\D_x} \Dd^+ .
\end{equation*}
Quillen's super-formalism would demand a minus sign to be
introduced when an odd operator passes over an odd-degree
differential form. We make this remark because there are numerous sign
discrepancies in this work and in other works, for example,
\cite{BBeN-GeE-VeM92}.

In the spirit of \cite{BiJ-FrDI86,IFrD87} it would be more usual to define
\[
\b\omega(s) =
 \bTr \Big( (\Dd^- \Dd^+)^{- \frac{s}{2} - 1} \Dd^-
\dot{\Dd}^+ \Big).
\]
This connection one form defines a connection on the $b$-determinant line bundle.  However, the connection is \emph{not} unitary. The reason the usual definition of the Bismut-Freed connection one form fails to define a unitary connection is that the $b$-trace does not vanish on commutators; the way we have forced symmetry in the definition of the connection one form gives unitarity.

There is a canonical non-vanishing section $\det \Dd^+$ of the $b$-determinant line bundle $\mathcal{L} \to \S^1$, which satisfies
\[
\bnabla^{\det} \det (\Dd^+) = \b \omega \otimes \det (\Dd^+),
\]
where $\bnabla^{\det}$ denotes the $b$-Bismut-Freed connection on $\mathcal{L}$ and $\b \omega$ is the $b$-Bismut-Freed connection one form, which equals the regular value of $\b \omega(s)$ at $s = 0$. 

We are now ready to state the main result.

\begin{theorem}[Global Anomaly Formula]
\label{thm_GlobalAnomaly}
\label{thm-main}
The
normalized $b$-eta function $\beb(s,\Dd_t)$ extends smoothly to
$t = 0$ and in the convergence of meromorphic functions of $s \in \C$,
    \begin{equation*}
        \lim_{t\to0} i\pi\,^b\overline{\eta}(s,\Dt) = \int_{\S^1} \b \omega (s)  - \frac{\Gamma(\frac{s+1}{2})}{4\sqrt{\pi}}\int_{\S^1} \Tr((\Dd^2)^{-\frac{s+1}{2}}\Dd G^* \dot{G}).
    \end{equation*}
\end{theorem}
Within the integral of the last term, there is a modified eta function \eqref{bebintro} with an extra factor $G^{*} \dot{G}$. If $\Dd^+$ is a family of generalized Dirac operators, this term will vanish. Thus one can think of the last term as measuring the `non-Clifford' nature of the connection $\nabla$.
\\

%% file: abcalc.tex
\section{The Adiabatic $b$-Calculus}
\label{sec_abcalc}

\subsection{Adiabatic b-Geometry}
\label{sec_abgeo}
Let $M$ be a compact manifold with boundary, $B$ a closed manifold of dimension $n$, and suppose $\phi: M\to B$ is a fiber bundle with fibers diffeomorphic to a manifold with boundary $Y$ of dimension $m$. For simplicity, assume $\partial M$ and $\partial Y$ are connected. The map $\phi$ induces a map on the tangent spaces $d\phi: TM \to TB$. The \textit{vertical subbundle} of $TM$ is defined as
$T(M/B) :=\ker(d\phi).$

    An \textit{adiabatic $b$-vector field}, or \textit{$ab$-vector field} for short, is a smooth family of $b$-vector fields
    $v:[0,\infty) \to \mathcal{V}_b(M)$
    such that 
    $v(0) \in C^\infty(T(M/B)).$
    This last condition means that $v(0)$ is tangent to the fibers at each $x\in B$. We let $\mathcal{V}_{ab}(M)$ denote the space of $ab$-vector fields on $M$. These vector fields can be realized as sections of a vector bundle, the \textit{$ab$-tangent bundle} $^{ab}TM$ over $M\times \R_+$. We will also need the \textit{$ab$-cotangent bundle} $^{ab}T^*M := (^{ab}TM)^*$. 

Now, we will construct the \textit{$ab$-double space} $M^{2}_{ab}$. This will be analogous to the construction of $M^2_a$ as in \cite{MorS04} and \cite{HomologyAndAdiabatic}. Rather than blowing up the fiber diagonal $\Delta_\phi$ inside the space $M^2\times \R_+$, we will blow up the $b$-fiber diagonal inside the space $M^2_b\times \R_+$. Recall that for any fiber bundle $M \rightarrow{B}$, the fiber diagonal is given by the preimage of the diagonal of $B$ under the map $(\phi,\phi):M \times M \rightarrow{B \times B}$:
\begin{equation*}
    \Delta_\phi = (\phi,\phi)^{-1}\Delta_B.
\end{equation*}

To define the \textit{$b$-fiber diagonal}, we will take the above fiber diagonal and consider its pull-back under the blow down map $\beta:M^2_b \rightarrow{M\times M}$:
    \begin{equation*}
        \Delta_{\phi,b} := \overline{\beta_b^{-1}((\Delta_\phi) \setminus(\partial M)^2)}.
    \end{equation*}
Now we will blow up this diagonal in a way that is similar to what is done in the standard adiabatic case, 
\begin{equation*}
    M^2_{ab} := [M^2_b \times \R_+;\,\Delta_{\phi,b}\times\{0\} ].
\end{equation*}
The blow down map that will be the most important for us will be the composition $\beta_{ab}$ given by blowing down both the adiabatic and boundary blow ups that were performed.
\begin{equation*}
    \begin{tikzcd}
M^2_{ab} \arrow[rr,"\beta_{ab}"] \arrow[dr]
&& M^2 \times \R_+ \\
& M^2_b \times \R_+ \arrow[ur]
\end{tikzcd}
\end{equation*}
We define the ${ab}$-diagonal in this space as the pull-back of the $M$ diagonal,
\begin{equation*}
    \Delta_{ab} := \overline{\beta^{-1}_{ab}((\Delta_M \setminus (\partial M)^2) \times(0,\infty))}.
\end{equation*}
Within the space $M^2_{ab}$ there are several boundary hypersurfaces in general. The two corresponding to the two blow ups performed will be of special importance to us. These we denote by $\tf$ and $\ff$, for the \textit{temporal front face} and the \textit{boundary front face} respectively:
\begin{equation*}
    \tf := \overline{\beta^{-1}_{ab}(\Delta_\phi\setminus(\D M)^2)\times\{0\}},\qquad \ff:= \beta^{-1}_{ab}((\D M)^2 \times \R_+).
\end{equation*}

Finally, we need to define the appropriate density bundle from which our operators will take values in. First consider the line bundle $\Omega(^{ab}T^*M)$.
Let $\pi_R$ denote the projection $M \times M \to M$ onto the right factor.
    The \textit{$ab$-density bundle} is defined by considering two consecutive pullback bundles, $^{ab}\Omega := \beta^*_{ab} \pi^*_R \Omega(^{ab}T^*M)$. We can illustrate this more clearly with the following diagram.

\begin{equation*}
    \begin{tikzcd}
^{ab}\Omega \arrow[r] \arrow[d]
& \pi^*_R\Omega(^{ab}T^*M) \arrow[r] \arrow[d]
& \Omega(^{ab}T^*M) \arrow[d] \\
M^2_{ab} \arrow[r, "\beta_{ab}"]
& M^2 \times \R_+ \arrow[r, "\pi_R"]
& M \times \R_+
\end{tikzcd}
\end{equation*}

\subsection{Adiabatic b-Operators}
\label{sec_abops}

\begin{definition}[\textbf{Adiabatic b-Pseudodifferential Operators}]
    For each $j\in\R,$ the space $\Psi^j_{ab}(M)$ is the collection of $^{ab}\Omega$-valued distributions on $M^2_{ab}$ that:
    \begin{enumerate}
        \item[(i)] are classical conormal along $\Delta_{ab}$ of order $j$,
        \item[(ii)] are smooth up to the boundary hypersurfaces $\tf$ and $\ff$,
        \item[(iii)] are rapidly vanishing at all other boundary hypersurfaces of $M^2_{ab}$.
    \end{enumerate}
We will also need the collection of operators defined by $\Psi\ab^{j,k} := t^{-k}\Psi\ab^j$.
\end{definition}

Note that for a fixed $t$, the operators are standard $b$-operators in $\Psi_b(M)$. We will now proceed to give a very explicit, coordinate-dependent description of the key objects that have been defined so far. Away from the boundary, operators in $\Psi_{ab}$ are identical to the operators within the adiabatic calculus $\Psi_a$. Near the boundary, let us fix coordinates $(x,y,z)$ where $x\in \R^n$ represents the base coordinates, $y\in\R^{m-1}$ represents the interior coordinates of the fibers, and $z\in\R^+$ is the boundary coordinate. If $A\in\Psi^j_{ab}(M)$ and $u\in C^\infty(M)$, we can think of $A$ as an operator acting on $u$ via
\begin{equation*}
    Au(t,x,y,z) = \frac{1}{t^n}\int K(t,x,y,z,x',y',z') u(x',y',z')\,dx'dy'\frac{dz'}{z'},
\end{equation*}
where the pullback of $K$ by $\beta_{ab}$ is conormal along $\Delta_{ab}$. To find coordinates that are normal along $\Delta_{ab}$, we simply combine what is done in the $b$-calculus case and the adiabatic calculus case. We let
\begin{equation}
\label{eq_zetacoord}
    \zeta = \frac{x-x'}{t}, \qquad u = y-y', \qquad v=  \log\left(\frac{z}{z'}\right).
\end{equation}
Hence,
\begin{equation*}
    \beta_{ab}(t,x,y,z,\zeta,u,v) = (t,x,y,z,x-t\zeta,y-u,z e^{-v}).
\end{equation*}
Since $\Delta_M = \{x=x', y=y', z=z'\}$ in $M$, it follows that $\Delta_{ab} = \{\zeta = u = v = 0 \}$ in $M^2_{ab}$.

Since $\mathcal{K} := \beta^*_{ab}K$ is conormal along $\Delta_{ab}$, $Au$ can be locally expressed as
\begin{multline*}
    Au(t,x,y,z) = \\
    \frac{1}{t^n}\int e^{\frac{i}{t}(x-x')\cdot\xi} e^{i(y-y')\cdot\eta} \left(\frac{z}{z'}\right)^{i\lambda} a(t,x,y,z,\xi,\eta,\lambda) u(x',y',z')\,\dbar\xi\dbar\eta\dbar\lambda dx' dy' \frac{dz'}{z'}
\end{multline*}

The explicit symbol estimates are
\begin{equation*}
    |\D_\xi^I \D_\eta^J \D_\lambda^\alpha \D_x^K \D_y^L \D_z^\beta \D_t^\gamma a| \leq C(1+|\xi|+|\eta|+|\lambda|)^{j-|I|-|J|-\alpha},
\end{equation*}
uniformly in $x,y,z,$ and $t$.

Proving that $\Psi_{ab}(M)$ is graded and that composition is well-defined and respects the grading can be done by following the standard arguments. In particular, in order to prove that composition holds the \textit{$ab$-triple spaces} can be constructed. Details on the triple spaces for $\Psi_b$ can be found in \cite{GrieserBasics} and \cite{BMeR93}, and for $\Psi_a$ in \cite{MorS04}.

Of course, the most important examples of $ab$-operators will be $ab$-differential operators. An \textit{$ab$-differential operator} on $M$ of order $j\in\R$ is a sum of compositions of $ab$-vector fields
    \[A = \sum_{k=0}^N f_k v_{k,1}v_{k,2},\cdots v_{k,\ell_k}\]
    with $\ell_k \leq j$. The space of ${ab}$-differential operators of order $j$ is denoted $\diff\!\ab^j(M)$.
In local coordinates near the boundary, we can express $A\in\diff\!\ab^j(M)$ as
\begin{equation}
    A = \sum_{|I|+|J|+k \leq j} a_{I,J,k}(t,x,y,z)(t\D_x)^I(\D_y)^J(z\D z)^k,
\label{abdiff}
\end{equation}
where each $a_{I,J,k}$ is smooth in all its variables. The \textit{$ab$-principal symbol} $\absym(A)$ of an operator $A\in \Psi\ab^j$ is defined locally as the top homogeneous term of the classical expansion of the symbol of $A$.

For an $ab$-differential operator given by local coordinates as in \eqref{abdiff}, the principal symbol will have the form
\begin{equation}
    \absym(A)(t,x,y,z,\xi,\eta,\lambda) = \sum_{|I|+|J|+k = j}a_{I,J,k}(t,x,y,z)(i\xi)^I(i\eta)^J(i\lambda)^k.
    \label{absymb}
\end{equation}
As one would expect, the $ab$-principal symbol of an operator is globally well-defined as a function on $^{ab}T^*M$. We say an operator is \textit{elliptic} if it is invertible as a function on $^{ab}T^*M\setminus\{0\}$. Also, by following the standard arguments we have the following theorem.
\begin{theorem}
\label{thm_abSES}
    There exists a short exact sequence
    \begin{equation*}
        \begin{tikzcd}
0 \arrow[r]
& \Psi\ab^{j-1}(M) \arrow[r, hook]
& \Psi\ab^j(M) \arrow[r, "\absym"]
& S^{[j]}(^{ab}T^*M) \arrow[r]
& 0
\end{tikzcd}.
    \end{equation*}  
\end{theorem}

\subsection{The Normal Operator}
\label{sec_normalop}

Now we will introduce the normal operator of an adiabatic $b$-operator. In fact, there are \textit{two} families of normal operators present, one corresponding to $\tf$ and one corresponding to $\ff$. We will call the operators corresponding to the temporal front face the \textit{normal operators}, and we will call the operators corresponding to the boundary front face the \textit{indicial operators}.

\begin{definition}
    Let $\mathcal{K}$ be the distributional kernel on $M\ab^2$ of an operator $A\in\Psi\ab^j(M)$. The \textit{normal operator} of $A$ is given by the Fourier transform of the kernel $\mathcal{K}$ along the temporal front face $\tf$, while keeping all other variables fixed. In other words,
    \[N(A) := \mathcal{F(\mathcal{K
    |_{\tf}}}) = \int_{-\infty}^\infty e^{-i\zeta \cdot \xi} \mathcal{K}(0,x,y,z,\zeta,u,v)\, d\zeta.\]
\end{definition}
This integral is always absolutely convergent since $\mathcal{K}$ is assumed to be smooth on $\tf$ and to vanish rapidly as $|\zeta| \to \infty$.

An important example is the normal operator of an $ab$-differential operator $A \in \diff\!\ab^j(M)$. Recall from \eqref{abdiff} that in local coordinates the derivatives in the base directions vanish as $t\to 0$. For simplicity, we will omit the fiber variables and write
\begin{equation*}
    A = \sum_{|I|\leq j}b_I(t,x)(t\D_x)^I,
    \label{abdiff td_x}
\end{equation*}
where the smooth functions $b_I(t,x)$ take values in the space of $b$-differential operators on the fiber $M_x$. It follows that
\begin{equation}
    \begin{split}
    N(A)(x,\xi) &= \sum_{|I| \leq j}b_I(0,x)(i\xi)^I.
                \label{Normal D_x}
    \end{split}
\end{equation}
The following theorem is clear from the definition of $N$:
\begin{theorem}
\label{thm_normalSES}
    Recall that $\Psi\ab^{j,k}(M) = t^{-k}\Psi\ab^j(M)$. There exists a short exact sequence
    \begin{equation*}
        \begin{tikzcd}
            0 \arrow[r]
            & \Psi\ab^{j,k-1} \arrow[r, hook]
            & \Psi\ab^{j,k} \arrow[r, "N"]
            & S^{j}(T^*B; \Psi_b^j(Y)) \arrow[r]
            & 0
    \end{tikzcd}.
    \end{equation*}
\end{theorem}
It is also a straightforward exercise to show that $N(AB) =  N(A)N(B)$.

\begin{proposition}
\label{prop_bsymnormal}
    For $A\in \Psi\ab^j(M),$
    \begin{equation*}
        ^b\sigma(N(A)) =\, \absym(A)|_{t=0},
    \end{equation*}
    where the $b$-symbol is taken with the base variables and covariables fixed. 
\end{proposition}
\begin{proof}
    For $ab$-differential operators this is obvious in light of \eqref{absymb} and \eqref{Normal D_x}. For a general $ab$-operator, the kernel restricted to $\tf$ is given by
    \begin{equation*}
        \int e^{\frac{i}{t}(x-x')\cdot\xi} e^{i(y-y')\cdot\eta} \left(\frac{z}{z'}\right)^{i\lambda} a(0,x,y,z,\xi,\eta,\lambda)\,\dbar\xi\dbar\eta\dbar\lambda.
    \end{equation*}
    To find $N(A)$, we take the Fourier transform in $\zeta = \frac{1}{t}(x-x'),$
    \begin{equation*}
        \int e^{i(y-y')\cdot\eta} \left(\frac{z}{z'}\right)^{i\lambda} a(0,x,y,z,\xi,\eta,\lambda)\,\dbar\eta\dbar\lambda.
    \end{equation*}
    For a fixed $(x,\xi)$, $^b\sigma(N(A)) = a_j(0,x,y,z,\xi,\eta,\lambda)$, the top homogeneous term of the classical symbol $a$. But this is precisely $\absym(A)|_{t=0}$.
\end{proof}

The normal operator acts like a principal symbol for the kernels at the temporal front face. In our study of the adiabatic limit, we will also need a corresponding \textit{full symbol} at the temporal front face. We denote this symbol by $q$. The idea is to take the Fourier transform of each coefficient function of the Taylor series of the kernel at $\tf$, as opposed to taking the Fourier transform of just the constant term at $\tf$ as we did for $N$.

\begin{definition}
    Suppose $B = \S^1$ and suppose $x$ is the standard coordinate function that parametrizes $\S^1$. Let $\xi$ be its dual variable. Recall from equation \eqref{eq_zetacoord} that $\zeta := \frac{1}{t}(x -x').$ Let $A$ be an $ab$-operator given by the kernel
    \begin{equation*}
        \mathcal{K}(t,x,\zeta) = \int e^{i\zeta\cdot\xi}q(t,\zeta,\xi)\,\dbar\xi d\zeta,
    \end{equation*}
    with $q$ being a classical symbol in $\xi$ which takes on values in the space of $b$-operators on the fiber $Y$. By definition, $q$ is smooth at $\tf$ and has expansion
    \begin{equation*}
        q(t,\zeta,\xi) \sim q_0(\zeta,\xi)+ q_1(\zeta,\xi)t + q_2(\zeta,\xi)t^2 \cdots.
    \end{equation*}
    We define the \textit{full temporal symbol} $q(A)$ by the symbol described above. Note that the constant term $q_0(\zeta,\xi)$ is just the normal operator $N(A)$.
\end{definition}

The main property of the full symbol that we will use is the following composition formula. This is analogous to \cite[Lemma~15]{MorS04} for the adiabatic case.
\begin{theorem}
\label{thm_fullsymcomp}
    Let $A,B\in \Psi\ab(M)$. Then
    \begin{equation*}
        q(AB) \sim  \sum_{\alpha}\frac{(-i)^{|\alpha|}}{\alpha!}(\partial_\xi)^\alpha q(A) \cdot(t\partial_x)^\alpha q(B).
    \end{equation*}
\end{theorem}
\begin{proof}
    The classical symbol composition theorem is well known and the theorem follows from this and recognizing that $\partial_\zeta = t\partial_x$.
\end{proof}
Now we turn to defining the $b$-trace functional. Following \cite{BMeR93}, we begin by defining an extension of the integral. For all $\phi\in C^\infty(M;\,^b\Omega)$, the following \textit{$b$-integral} is well-defined:
\begin{equation*}
     \bint_M \phi := \lim_{\epsilon \to 0} \left(  \int_{z>\epsilon}\phi + \log{\epsilon} \int_{\partial M} \phi|_{\partial M}    \right).
\end{equation*}
Note that the $b$-integral depends on the boundary defining function chosen. Locally, the $b$-integral of a $b$-density $\phi = a(x,y,z)\,dxdy \frac{dz}{z}$ can be represented by
\begin{equation*}
    \int(a(x,y,z) - a(x,y,0))\,dxdy\frac{dz}{z}.
\end{equation*}
We can now define the \textit{$b$-trace} for an operator $A\in \Psi^{-\infty}_b(M)$ as follows:
\begin{equation*}
    \bTr(A) := \bint_M{(A|_{\Delta_b})}.
\end{equation*}
\begin{theorem}
\label{thmTrace}
    Suppose the base space is the circle $B = \S^1$ and suppose $x$ is the standard coordinate function that parametrizes $\S^1$. Let $\xi$ be its dual variable. Let $A\in \Psi\ab^k$ with $k<-\dim{M}$. Then 
    \begin{equation*}
        \lim_{t\to 0}  t\ \b\Tr(A) = \int_{T^*B} \b\Tr(N(A)).
    \end{equation*}
\end{theorem}

\begin{proof}
    The condition $k < -\dim{M}$ is to guarantee that the following integrals are well-defined. Using a partition of unity, it suffices to prove the theorem in local coordinates. Observe that
    \begin{equation*}
        \begin{split}
            \lim_{t\to 0} t \b\Tr(A)                    &= \lim_{t \to 0} \bint_M \int a(t,x,y,z,\xi,\eta,\lambda)\,\dbar\xi\dbar\eta\dbar\lambda dxdy\frac{dz}{z}.
        \end{split}     
    \end{equation*}
    Using the local description of the $b$-integral, this is equivalent to
    \begin{align*}
    \lim_{t\to 0} t \b\Tr(A) &=
        \lim_{t\to 0}   \int a(t,x,y,z,\xi,\eta,\lambda) - a(t,x,y,0,\xi,\eta,\lambda)
        \,\dbar\xi\dbar\eta\dbar\lambda dxdy\frac{dz}{z} \\
        &= \int a(0,x,y,z,\xi,\eta,\lambda) - a(0,x,y,0,\xi,\eta,\lambda)
        \,\dbar\xi\dbar\eta\dbar\lambda dxdy\frac{dz}{z} \\
        &= \int\left( \int a(0,x,y,z,\xi,\eta,\lambda) - a(0,x,y,0,\xi,\eta,\lambda)\, \dbar\eta \dbar\lambda dy\frac{dz}{z} \right)\dbar\xi dx.
    \end{align*}
    But the integrand is just the $b$-integral of the normal operator along the fiber $Y$,
    \begin{align*}
        \int_{T^*B} \left( \bint_Y \left( \int a(0,x,y,z,\xi,\eta,\lambda)\,\dbar\eta\dbar\lambda \right) dy\frac{dz}{z} \right)\, dx \dbar\xi 
        = \int_{T^*B} \b\Tr(N(A)).
    \end{align*}
\end{proof}
We end this section with the comment that one can generalize the content above to construct $\Psi_{ab}(M;E)$ for some vector bundle $E$. Likewise, the extension of the calculus to include complex orders can also be done. Both extensions can be carried out in the standard way, which we omitted for notational simplicity.

%% file: witten.tex
\section{The Global Anomaly Formula for Manifolds with Boundary}
\label{sec_witten}

We will assume that $\phi:M\to B$ is a fiber bundle and $M$ is a manifold with boundary. In this section the base space will be the circle $B = \S^1$. The fibers are manifolds with boundary and will be represented by $Y$. Let $E\to M$ be a Hermitian vector bundle over $M$ where $E = E^+ \oplus E^-$. Let $\Dd^+\in \diff^1_{b,\phi}(M;E^+,E^-)$ be a family of \textit{invertible}, first order, $b$-differential operators. Thus for each $x\in B$,
\begin{equation*}
    \Dd_x^+: C^\infty(M_x;E^+|_{M_x}) \to C^\infty(M_x; E^-|_{M_x})
\end{equation*}
is an invertible, first order, $b$-differential operator on the fibers $M_x \cong Y$. Let $\Dd^- = (\Dd^+)^*$. Let
\begin{equation*}
    \nabla: C^\infty(M;E) \to C^\infty (M; T^*M \otimes E) 
\end{equation*}
be a unitary $\Z_2$-graded connection on the vector bundle $E\to M$ and let
\begin{equation*}
    \tnabla: C^\infty(M;E) \to C^\infty(M, \phi^*T^*B \otimes E)
\end{equation*}
be the induced connection as constructed, for example, in \cite{MorS04}. Let $\D_x$ be the standard angular vector field on $\S^1$.

We assume the following Dirac condition on the indicial operator of $\Dd^+$: the indicial operator of $\Dd^+$ can be written as
\begin{equation*}
\label{DdnearM} I(\Dd^+) = iG \lambda  + P,
\end{equation*}
where $P \in
\mathrm{Diff}^1_{b,\phi|_{\partial M}}(\partial M,E^+,E^-)$, and the map $G: E^+|_{\partial M} \to E^-|_{\partial M}$ is a unitary bundle map satisfying
\begin{equation*} \label{GB}
G^* P = P^* G.
\end{equation*}
These conditions guarantee that
\[
I(\Dd^- \Dd^+)(\lambda) = \lambda^2 + P^* P.
\]
Alternatively, we have
\begin{equation*}
\label{DH} I(\Dd^+) = G\, \big( i \lambda + \Dd \big)
\end{equation*}
where
$\Dd  = G^* P\in \mathrm{Diff}^1_{b,\phi|_{\partial M}}(\partial M,E^+|_{\partial M})$,
and is formally self-adjoint and invertible on the fibers. We shall call $\Dd$ the \emph{induced} operator on $\partial M$. Since $\Dd$ is Fredholm; we can form its \emph{eta function}. From \cite{LauR-MorS02,LoP01,LoPII01}, we know that this defines a smooth function on the base $\S^1$ and is a meromorphic function of  $s \in \C$ with at most simple poles at $\dim Y, \dim Y - 1, \dim Y - 2,\ldots$; using the arguments from \cite{refv} it also follows that the origin is always a regular value. We can define the normalized eta function of the boundary:
\begin{equation*} \label{betadef}
\overline{\eta}(s, \D M) = \frac{\Gamma\left( \frac{s + 1}{2}
\right)}{\sqrt{\pi}} \Tr \Big( (\Dd^2)^{-
\frac{s + 1}{2}} \Dd \Big),
\end{equation*}

in which we take the fiber-wise trace for each $x\in\S^1$. In the above expression and for what follows, we will use the heat operator to represent complex powers. Details on this can be found in \cite{complex powers}. For $s\in\C$ with $\Re(s) > 0$,
\begin{equation*}
    (\Dd^2)^{-s} = \frac{1}{\Gamma(s)} \int_0^\infty \t^s e^{-\t \Dd^2}\,\frac{d\t}{\t}.
\end{equation*}

\subsection{The $\Dd_t$ Operator}
For $t\geq0$ define an operator $\Dd_t: \dot{C}^\infty(M;E) \to \dot{C}^\infty(M;E)$ by
\begin{equation*}
    \Dd_t = \begin{pmatrix}
        -it\tnabla_{\D_x} & \Dd^- \\
        \Dd^+             & it\tnabla_{\D_x}
    \end{pmatrix},
    \label{Dd_t}
\end{equation*}
where the block matrix form is with respect to the grading $E=E^+ \oplus E^-$.  $\Dt$ is a formally self-adjoint, elliptic operator in $\diff_{ab}^1(M,E)$.

The basic argument in the following proposition has been adapted from \cite[Proposition~19]{MorS04}.
\begin{proposition}
\label{prop_invert}
    There exists an $\epsilon >0$ such that $\Dd_t: H^1_b(M;E) \to L^2_b(M;E)$ is invertible for all $0 < t< \epsilon$.
\end{proposition}

\begin{proof}
    Since  $\Dd^+\Dd^-$ and $\Dd^-\Dd^+$ are self adjoint and positive, the diagonal entries of $N(\Dd_t^2)$ are invertible, hence $N(\Dd^2_t)$ is itself invertible. It follows that $N(\Dd_t)$ is also invertible.

    By Theorem~\ref{thm_normalSES}, we can define an operator $B_0$ so that $N(B_0) = N(\Dd_t)^{-1}$. From this and Proposition \ref{prop_bsymnormal} we see that
    \begin{equation*}
        I =\, \b\sigma(N(\Dd_t B_0)) = \, \absym(\Dd_t)|_{t=0}\,\absym(B_0)|_{t=0}.
    \end{equation*}
    By Theorem \ref{thm_abSES} we can take $\absym(B_0)$ to be independent of $t$, and it is clear that $\absym(\Dt)$ is also independent of $t$. It follows that $\absym(B_0) =\, \absym(\Dd_t)^{-1}$. Hence we have that
    \begin{equation*}
        \Dd_t B_0 = I - R_0,
    \end{equation*}
    where $B_0 \in \Psi\ab^{-1,0}(M;E)$ and $R_0 \in \Psi\ab^{-1,-1}(M;E)$. Let
    \begin{equation*}
        B_1 = B_0(I+R_0+R_0^2+R_0^3 + \cdots)
    \end{equation*}
    be defined by asymptotically summing the powers of the remainder term. It follows that
    \begin{equation*}
            \Dd_t B_1  = I - R,
    \end{equation*}
    where $B_1\in \Psi\ab^{-1,0}$ and $R\in\Psi\ab^{-\infty,-\infty}$. Thus for any $k\in\Z$, there exists an $\widetilde{R}\in\Psi\ab^{-\infty}$ such that $R = t^k\widetilde{R}$. Since $\widetilde{R}(t) \in \Psi^{-\infty}_b(M;E)$, it follows that $\widetilde{R}(t)$ is bounded as an operator on $L^2_b(M;E)$. See \cite{BMeR93} for details on the $b$-Sobolev spaces. Since $\|R\| = t^k\|\widetilde{R}(t)\|$, for sufficiently small $t$ we have that $\|R\| <1$. By the standard Neumann series argument, the inverse for $\Dt$ can finally be constructed.
\end{proof}

\subsection{Preliminary Propositions}
An important fact about the $b$-trace is that commutators in general \textit{do} \textit{not} vanish under them. The exact behavior of the $b$-trace functional applied to a commutator is described by the following theorem, the \textit{trace defect formula} \cite{BMeR93}. This formula is indispensable in our proof of the global anomaly formula, as well as in the proofs of the following propositions.
\begin{theorem}
\label{thm_tracedefect}
    Suppose $A\in \Psi^m_b(M)$ and $B\in\Psi^{-\infty}_b(M)$. If $\dbar\lambda = \frac{1}{2\pi}\,d\lambda$, then
    \begin{equation*}
        \bTr[A,B] = i \int_{-\infty}^\infty \Tr(\partial_\lambda I(A) \circ I(B))\,\dbar\lambda.
    \end{equation*}
\end{theorem}
Our proof of Theorem $\ref{thm-main}$ will closely follow the proof given by Moroianu in \cite{MorS04} of the closed manifold case. The main difference in the boundary case is that the $b$-trace of a commutator fails to vanish, as described by Theorem \ref{thm_tracedefect}. These failures will result in several remainder terms that will ultimately lead to the last term of Theorem \ref{thm-main}. For the sake of preserving the thread of the main argument, we will present the simplification of these remainder terms here as preliminary lemmas. We also state the variation of the eta function at the boundary and of the $b$-zeta function of $\Dd^-\Dd^+$. A key tool will be the following fact which can be found in \cite{BBeN-GeE-VeM92} as equation (2.6) for an elliptic operator $A$:
    \begin{equation}
    \label{lem_BGVpartialxexp}
        [\tnablax, e^{-\tau A}] = -\int_0^\t e^{-(\tau-r)A} [\tnablax,A] e^{-rA}\,dr.
    \end{equation}
We will also use that \[\partial_x\Tr\left((\Dd^2)^{-\frac{s+1}{2}}\Dd\right) = \Tr[\tnablax,(\Dd^2)^{-\frac{s+1}{2}}\Dd],\] and that \[\partial_x\bTr\left(\left(\Dd^-\Dd^+\right)^{-\frac{s}{2}}\right) = \bTr[\tnablax,\left(\Dd^-\Dd^+\right)^{-\frac{s}{2}}].\]
The first equation follows from the construction of $\tnabla$ in \cite{MorS04}. The second equation is proved in an analogous way, with the only difference being the failure of the $b$-trace to vanish on commutators, particularly with the components of $\tnabla$ that act only along the fiber. However, by a straightforward application of Theorem \ref{thm_tracedefect} along with our choice of a horizontal $b$-connection that is compatible with the boundary defining function, the identity holds nonetheless. Both identities hold initially for $\Re(s)$ sufficiently large and extend meromorphically.

\begin{proposition}
\label{lem_bEta}
Recall that
$\overline{\eta}(s, \D M) : = \frac{\Gamma\left( \frac{s + 1}{2}
\right)}{\sqrt{\pi}} \Tr \Big( (\Dd^2)^{-
\frac{s + 1}{2}} \Dd \Big)$.
    Then \[\D_x\, \overline{\eta}(s,\D M) = -s\frac{\Gamma(\frac{s+1}{2})}{\sqrt{\pi}}\Tr\left(  (\Dd^2)^{-\frac{s+1}{2}}\dot{\Dd}  \right).
    \]
\end{proposition}
\begin{proof} Since $[\tnablax, \cdot]$ is a derivation, we have that
    \begin{align*}
        \D_x \Tr\left(  (\Dd^2)^{-\frac{s+1}{2}}\Dd  \right) &= \Tr\left[\tnablax,  (\Dd^2)^{-\frac{s+1}{2}}\Dd\right] \\
        & = \frac{1}{\Gamma(\frac{s+1}{2})} \int_0^\infty \t^{\frac{s+1}{2}} \Tr\left( \left[  \tnablax, e^{-\t\Dd^2}  \right] \Dd \right) \, \frac{d\t}{\t} + \Tr\left(   (\Dd^2)^{-\frac{s+1}{2}}\dot{\Dd}    \right).
    \end{align*}
    By $\eqref{lem_BGVpartialxexp}$ the first term becomes
    \begin{align*}
       & \frac{-1}{\Gamma(\frac{s+1}{2})}\int_0^\infty \int_0^\t \t^{\frac{s+1}{2}} \Tr\left( e^{-(\t-r)\Dd^2}[\tnablax,\Dd^2] e^{-r\Dd^2} \Dd \right)\, dr\frac{d\t}{\t} \\
       =&  \frac{-1}{\Gamma(\frac{s+1}{2})}\int_0^\infty \int_0^\t \t^{\frac{s+1}{2}} \Tr\left( e^{-\t\Dd^2}[\tnablax,\Dd^2] \Dd \right)\, dr\frac{d\t}{\t} \\
       =& -\frac{s+1}{2}\Tr \left( (\Dd^2)^{-\frac{s+3}{2}}[\tnablax,\Dd^2]\Dd \right) \\
       =& -(s+1) \Tr\left( (\Dd^2)^{-\frac{s+1}{2}} \dot{\Dd}\right),
    \end{align*}
    from which the proposition follows.
\end{proof}

\begin{proposition}
\label{lem_bzeta}
    Define the normalized $b$-zeta function of $\Dd^-\Dd^+$ by
    
\[^b\overline{\zeta}(s,\Dd^-\Dd^+) = \Gamma\left(\frac{s}{2}\right) \bTr\left(\left(\Dd^-\Dd^+\right)^{-\frac{s}{2}}\right).\]
    Then
    $
        \D_x\, ^b\overline{\zeta}(s,\Dd^-\Dd^+) = -\Gamma\left(\frac{s}{2}+1\right) \bTr\left(\left(\Dd^-\Dd^+\right)^{-\frac{s}{2}-1}[\tnablax,\Dd^-\Dd^+]\right)
    $.
\end{proposition}
\begin{proof}
    By \eqref{lem_BGVpartialxexp}, we have
    \begin{equation*}
        \partial_x\, \b\overline{\zeta}(s,\Dd^-\Dd^+)
        = -\int_0^\infty \int_0^\t \t^\frac{s}{2} \bTr\left( e^{-(\t-r)\Dd^-\Dd^+}[\tnablax,\Dd^-\Dd^+]e^{-r\Dd^-\Dd^+} \right)\,dr\frac{d\t}{\t}.
    \end{equation*}
    We will commute the rightmost heat operator over to the front of the $b$-trace, but unfortunately we will pick up a remainder term since the $b$-trace of a commutator does not necessarily vanish. So the above becomes
    \begin{equation}
        -\int_0^\infty \int_0^\t \t^\frac{s}{2} \bTr\left( e^{-\t\Dd^-\Dd^+}[\tnablax,\Dd^-\Dd^+] \right)\,dr\frac{d\t}{\t} +R_0,
        \label{eq_bzetavariation}
    \end{equation}
    where
    \begin{equation*}
        R_0 = -\int_0^\infty \int_0^\t \t^\frac{s}{2} \bTr\left[ e^{-(\t-r)\Dd^-\Dd^+}[\tnablax,\Dd^-\Dd^+],e^{-r\Dd^-\Dd^+} \right]\,dr\frac{d\t}{\t}.
    \end{equation*}
    Leaving the remainder to the side for the moment, from \eqref{eq_bzetavariation} we can integrate out the $r$ variable and rewrite with complex powers to get
    \begin{equation*}
        -\Gamma\left(\frac{s}{2}+1\right) \bTr\left(\left(\Dd^-\Dd^+\right)^{-\frac{s}{2}-1}[\tnablax,\Dd^-\Dd^+]\right) +R_0.
    \end{equation*}
    All that remains to be shown is that $R_0 =0$. To this end, we will use the trace defect formula, Theorem \ref{thm_tracedefect}. Since both
    \begin{equation*}
        I\left( e^{-(\t-r)\Dd^-\Dd^+)}[\tnablax,\Dd^-\Dd^+] \right)(\lambda) = e^{-(\t-r)(\lambda^2+P^*P)}[\tnablax,P^*P]
    \end{equation*}
    and
    \begin{equation*}
        I\left( e^{-r\Dd^-\Dd^+} \right)(\lambda) = e^{-r(\lambda^2+P^*P)}
    \end{equation*}
    are even in $\lambda$, the integral in Theorem \ref{thm_tracedefect} vanishes and $R_0=0$.
\end{proof}

\begin{lemma}
\label{lem_R1} Let 
\begin{equation*}
    \alpha = \begin{pmatrix}
        (-it\tnablax)^2+\Dd^-\Dd^+ & 0 \\ 0 & (it \tnablax)^2+ \Dd^+\Dd^-
    \end{pmatrix}
\end{equation*}
and \begin{equation*}
    \beta = \begin{pmatrix}
        0 & -i\dot{\Dd}^- \\ i \dot{\Dd}^+ & 0
    \end{pmatrix}.
\end{equation*}
Define
    \begin{equation*}
    R_1 = -\frac{1}{\Gamma(\frac{s+1}{2})} \int_0 ^\infty \int_0^\tau \t^\frac{s+1}{2} \bTr[e^{-(\t-r)N(\alpha)}N(\beta)N(\Dt), e^{-r N(\alpha)}]\,dr \frac{d\t}{\t}.
\end{equation*}
Then
    \begin{equation*}
        \int_{T^*\S^1} R_1  = -\frac{1}{8i\pi} \int_{\S^1} \Tr((\Dd^2)^{-\frac{s+1}{2}}(G^*\dot{G}\Dd-G^*\dot{P}-\dot{P}^*G+\dot{G}^*P)) .
    \end{equation*}
\end{lemma}
\begin{proof}
Carrying out the matrix multiplication for the commutator terms in the above definition, we have that
\begin{multline*}
    e^{-(\t-r)N(\alpha)}N(\beta)N(\Dt)e^{-r N(\alpha)} = \\  \begin{pmatrix}
        -i e^{-(\t-r)(\xi^2+\Dd^-\Dd^+)}\dot{\Dd}^-\Dd^+e^{-r(\xi^2+\Dd^-\Dd^+)} & * \\
        * & i e^{-(\t-r)(\xi^2+\Dd^+\Dd^-)}\dot{\Dd}^+\Dd^- e^{-r(\xi^2+\Dd^+\Dd^-)}
    \end{pmatrix}
\end{multline*}
and
\begin{equation*}
    e^{-\t N(\alpha)} N(\beta) N(\Dt) 
    = \begin{pmatrix}
        -i e^{-\t(\xi^2+\Dd^-\Dd^+)}\dot{\Dd}^-\Dd^+ & * \\
        * & i e^{-\t(\xi^2+\Dd^+\Dd^-)}\dot{\Dd}^+\Dd^-
    \end{pmatrix},
\end{equation*}
where the asterisks represent the irrelevant off diagonal terms. Taking the $b$-trace of the commutator of the matrices and recombining the terms gives us
\begin{multline}
    \bTr[e^{-(\t-r)N(\alpha)}N(\beta)N(\Dt), e^{-r N(\alpha)}] \\
    = i e^{-\t\xi^2}( \bTr[e^{-(\t-r)\Dd^+\Dd^-}\dot{\Dd}^+\Dd^-,  e^{-r\Dd^+\Dd^-}] - \bTr[e^{-(\t-r)\Dd^-\Dd^+}\dot{\Dd}^-\Dd^+,  e^{-r\Dd^-\Dd^+}]) \\
    = i e^{-\t\xi^2}(T_1 - T_2),
    \label{eq_T1T2}
\end{multline}
where $T_1$ and $T_2$ are defined by the previous line. We will use Theorem \ref{thm_tracedefect} to simplify both $T_1$ and $T_2$. Beginning with $T_1$ we have
\begin{align*}
    T_1 &= i\int_{\R} \Tr\left(  \D_\lambda(e^{-(\t-r)(\lambda^2+PP^*)}(i\lambda\dot{G}+\dot{P})(-i\lambda G^* + P^*)) \circ e^{-r(\lambda^2+PP^*)}   \right)\,\dbar\lambda.
\end{align*}
Any terms in the integrand that are odd in $\lambda$ will vanish since the integral is over $\R$. After differentiating, expanding, and removing the odd terms we are left with
\begin{align*}
    T_1 &= - \int_\R \Tr \left(  e^{-\t(\lambda^2+PP^*)}\left(-2\lambda^2(\t-r)(\dot{G}P^*-\dot{P}G^*)+(\dot{G}P^*-\dot{P}G^*)\right)   \right)\,\dbar\lambda \\
    &= -\frac{r}{2\t^{3/2}\sqrt{\pi}} \Tr\left(e^{-\t PP^*}(\dot{G}P^*- \dot{P}G^*)\right).
\end{align*}
In a nearly identical way with the roles of $\Dd^-$ and $\Dd^+$ swapped, we also have
\begin{align*}
    T_2 = -\frac{r}{2\t^{3/2}\sqrt{\pi}} \Tr\left( e^{-\t P^*P}(\dot{P}^*G - \dot{G}^* P)  \right).
\end{align*}
Putting these terms back into \eqref{eq_T1T2} yields
\begin{multline*}
    \bTr[e^{-(\t-r)N(\alpha)}N(\beta)N(\Dt), e^{-r N(\alpha)}] \\ = \frac{r e^{-\t\xi^2} }{2i\t^{3/2}\sqrt{\pi}} \Tr\left(  e^{-\t PP^*}(\dot{G}P^* - \dot{P}G^*) - e^{-\t P^*P}(\dot{P}^* G - \dot{G}^* P) \right),
\end{multline*}
and thus $R_1$ becomes
\begin{equation*}
     \frac{1}{4i\sqrt{\pi}\Gamma(\frac{s+1}{2})}  \int_0^\infty -e^{-\t\xi^2}\t^{\frac{s}{2}+1} \Tr\left(  e^{-\t PP^*}(\dot{G}P^* - \dot{P}G^*) - e^{-\t P^*P}(\dot{P}^* G - \dot{G}^* P) \right) \,\frac{d\t}{\t}.
\end{equation*}
Integrating $\xi$ over $\R$ with $\dbar\xi = (2\pi)^{-1}\,d\xi$ gives
\begin{equation*}
    \frac{-1}{8i\pi\Gamma(\frac{s+1}{2})}  \int_0^\infty \t^{\frac{s+1}{2}} \Tr\left(  e^{-\t PP^*}(\dot{G}P^* - \dot{P}G^*) - e^{-\t P^*P}(\dot{P}^* G - \dot{G}^* P) \right) \,\frac{d\t}{\t}.
\end{equation*}
By the uniqueness of solutions to the heat equation, note that $e^{-\t PP^*} = G e^{-\t P^*P} G^*$. Using this fact, along with $G^* = G^{-1}$, observe that
\begin{align*}
    \Tr \left(  e^{-\t PP^*}(\dot{G}P^* - \dot{P}G^*)  \right) & = \Tr\left(  Ge^{-\t P^*P}G^*(\dot{G}P^* - \dot{P}G^*)  \right) \\
    &= \Tr\left(  e^{-\t P^*P}(G^*\dot{G}\Dd - G^*\dot{P}) \right).
\end{align*}
With this and with the fact that $e^{-\t P^*P} = e^{-\t P^* G G^* P} = e^{-\t \Dd^2}$, $\int R_1\,\dbar\xi$ becomes
\begin{multline*}
    \frac{-1}{8i\pi\Gamma(\frac{s+1}{2})}  \int_0^\infty \t^{\frac{s+1}{2}} \Tr\left(  e^{-\t \Dd^2}(G^*\dot{G}\Dd - G^* \dot{P}-\dot{P}^* G + \dot{G}^* P) \right) \,\frac{d\t}{\t} \\
    = \frac{-1}{8i\pi} \Tr\left(  (\Dd^2)^{-\frac{s+1}{2}}(G^*\dot{G}\Dd - G^* \dot{P}-\dot{P}^* G + \dot{G}^* P)   \right),
\end{multline*}
as desired.
\end{proof}

We need one more lemma before we prove the main result.

\begin{lemma}
\label{lem_R2R3}
Define
        \begin{align*}
    R_2 &= \bTr\left[(\xi^2+\Dd^+\Dd^-)^{-\frac{s+3}{2}},\dot{\Dd}^+\Dd^-\right], \\
    R_3 &= \bTr\left[\dot{\Dd}^+(\xi^2+\Dd^-\Dd^+)^{-\frac{s+3}{2}},\Dd^-\right].
\end{align*}
Then
    \begin{equation*}
        \int_{T^*\S^1}\frac{s+1}{2i}(R_2+R_3)\,\dbar\xi = \frac{1}{4\pi i } \int_{\S^1} \Tr \left(  (\Dd^2)^{-\frac{s+1}{2}}(G^* \dot{G} \Dd - 2G^* \dot{P})  \right).
    \end{equation*}

\end{lemma}
\begin{proof}
    The process of simplifying these two remainder terms will be similar to that of the previous lemma. Breaking down the complex powers using heat operators, the relevant $b$-trace expressions are $\bTr\left[ e^{-\t\Dd^+\Dd^-}, \dot{\Dd}^+\Dd^-\right]$ and $\bTr\left[ \dot{\Dd}^+e^{-\t\Dd^-\Dd^+}, \Dd^-\right]$, respectively. Applying Theorem \ref{thm_tracedefect} to the first of these gives
\begin{equation*}
    i \int_\R \Tr\left( \D_\lambda \left( e^{-\t(\lambda^2+PP^*)}\right) (i\lambda \dot{G}+\dot{P})(-i\lambda G^*+P^*) \right)\,\dbar\lambda.
\end{equation*}
After differentiating, expanding, and getting rid of the terms that are odd in $\lambda$ we have
\begin{equation*}
     \frac{1}{2\t^{1/2}\sqrt{\pi}} \Tr\left( e^{-\t PP^*}(\dot{G}P^* -\dot{P}G^*) \right) .
\end{equation*}
Similarly, applying the trace defect formula to the second $b$-trace above gives
\begin{equation*}
    i \int_\R \Tr\left( (-i\lambda G^* +P^*)\D_\lambda\left( (i\lambda\dot{G}+\dot{P})e^{-\t(\lambda^2+P^*P)} \right)  \right)\,\dbar\lambda.
\end{equation*}
After differentiating, expanding, and eliminating the odd terms in $\lambda$ we have
\begin{equation*}
     - \frac{1}{2\t^{1/2}\sqrt{\pi}} \Tr\left(G^*\dot{P} e^{-\t P^*P}\right)
\end{equation*}
Combining both $b$-trace terms together yields
\begin{multline*}
    \bTr\left[ e^{-\t\Dd^+\Dd^-}, \dot{\Dd}^+\Dd^-\right] + \bTr\left[ \dot{\Dd}^+e^{-\t\Dd^-\Dd^+}, \Dd^-\right] \\
    = \frac{1}{2\t^{1/2}\sqrt{\pi}}\Tr \left( e^{-\t PP^*}(\dot{G}P^* - \dot{P}G^*) - G^*\dot{P}e^{-\t P^*P}\right) \\
    = \frac{1}{2\t^{1/2}\sqrt{\pi}}\Tr \left( e^{-\t \Dd^2}(G^*\dot{G}\Dd - 2G^*\dot{P})\right).
\end{multline*}
 Thus we have
\begin{align*}
    R_2+R_3 &= \frac{1}{2\sqrt{\pi}\Gamma(\frac{s+3}{2})} \int_0^\infty \t^{\frac{s}{2}+1} e^{-\t\xi^2} \Tr \left( e^{-\t \Dd^2}(G^*\dot{G}\Dd - 2G^*\dot{P})\right) \,\frac{d\t}{\t}.
\end{align*}
Integrating this over $\xi$ gives
\begin{equation*}
    \int_{T^*\S^1}\frac{s+1}{2i}( R_2+R_3)\,\dbar\xi = \frac{1}{4\pi i} \int_{\S^1} \Tr\left((\Dd^2)^{-\frac{s+1}{2}}(G^*\dot{G}\Dd - 2G^*\dot{P})\right).
\end{equation*}
\end{proof}

\subsection{The Adiabatic Limit}

Since $\Dd_t$ is invertible for small $t$, its normalized $b$-eta function is well-defined:
\begin{equation*}
    \b\,\overline{\eta}(s,\Dd_t) = \frac{\Gamma(\frac{s+1}{2})}{\sqrt{\pi}}\bTr((\Dd_t^2)^{-\frac{s+1}{2}}\Dt).
\end{equation*}
Since $(\Dd_t^2)^{-\frac{s+1}{2}}\Dt\in \Psi\ab^{-s}(M)$, Theorem \ref{thmTrace} applies and as $t\to 0$ yields
\begin{multline}
     \bTr((\Dt^2)^{-\frac{s+1}{2}}\Dt)\\ \sim \int_{T^*\S^1} \frac{1}{t}\bTr(N((\Dt^2)^{-\frac{s+1}{2}}\Dt))+ \bTr(q_1((\Dt^2)^{-\frac{s+1}{2}}\Dt)) + O(t).
    \label{eqtrace}
\end{multline}
The Global Anomaly Formula, Theorem \ref{thm_GlobalAnomaly}, will be proved by showing that the $t^{-1}$ term vanishes in the limit, and that the constant term yields the expected formula. We begin with the coefficient of $t^{-1}$:
\begin{lemma}
\label{lem_btrnormal}
    \begin{equation*}
        \int_{T^*\S^1} \bTr(N((\Dt^2)^{-\frac{s+1}{2}}\Dt))= 0
    \end{equation*}
\end{lemma}
\begin{proof} With asterisks denoting the off diagonal terms,
    \begin{align*}
        N((\Dt^2)^{-\frac{s+1}{2}}\Dt) &= \begin{pmatrix}
            \xi(\xi^2+\Dd^-\Dd^+)^{-\frac{s+1}{2}} & * \\
            * & -\xi(\xi^2+ \Dd^+ \Dd^-)^{-\frac{s+1}{2}}
        \end{pmatrix}.
    \end{align*}
    Thus $\bTr(N((\Dt^2)^{-\frac{s+1}{2}}\Dt)) = \bTr(\xi(\xi^2+\Dd^-\Dd^+)^{-\frac{s+1}{2}}) - \bTr(\xi(\xi^2+ \Dd^+ \Dd^-)^{-\frac{s+1}{2}})$. Since this expression is odd in $\xi$, the lemma is proved.
\end{proof}

We have shown that $\lim\limits_{t\to 0} \bTr((\Dt^2)^{-\frac{s+1}{2}}\Dt)$ exists. We will now move our attention to the constant term in \eqref{eqtrace}. From Theorem \ref{thm_fullsymcomp} we have that
\begin{align*}
    q((\Dt^2)^{-\frac{s+1}{2}}\Dt)  &= q((\Dt^2)^{-\frac{s+1}{2}})(q(\Dt)) -it \partial_\xi q((\Dt^2)^{-\frac{s+1}{2}}) \tnablax(q(\Dt)) +O(t^2).
\end{align*}
If we expand each full symbol above up to the linear terms in $t$, we have that
\begin{multline*}
    q((\Dt^2)^{-\frac{s+1}{2}}\Dt) = N((\Dt^2)^{-\frac{s+1}{2}})N(\Dt) \\ + t( q_1((\Dt^2)^{-\frac{s+1}{2}})N(\Dt) - i\partial_\xi N((\Dt^2)^{-\frac{s+1}{2}}) \tnablax N(\Dt)) +O(t^2).
\end{multline*}
The constant term is just $N((\Dt^2)^{-\frac{s+1}{2}})N(\Dt)$ which, as we have just shown in the above Lemma, will vanish in \eqref{eqtrace}. The linear terms in $t$ in the above equation correspond to the constant term in \eqref{eqtrace}. This next Lemma shows that the second term of order $t$ in the above equation also vanishes in \eqref{eqtrace}.

\begin{lemma}
\label{lem_btrderivs}
    \begin{equation*}
        \int_{T^*\S^1}  \bTr(\partial_\xi N(\DD) \tnablax N(\Dt) )  = 0.
    \end{equation*}
\end{lemma}
\begin{proof}
    Since $e^{-\tau(\xi^2+\Dd^-\Dd^+)} = e^{-\tau\xi^2} e^{-\tau\Dd^-\Dd^+}$ by the uniqueness of solutions to the heat equation, observe that \begin{align*}
        \partial_\xi (\xi^2+\Dd^-\Dd^+)^{-\frac{s+1}{2}} &= \partial_\xi \frac{1}{\Gamma(\frac{s+1}{2})} \int_0^\infty \tau^{\frac{s+1}{2}} e^{-\tau(\xi^2+\Dd^-\Dd^+)}\,\frac{d\tau}{\tau} \\
        &= -2\xi\left(\frac{s+1}{2}\right) (\xi^2+\Dd^-\Dd^+)^{-\frac{s+3}{2}}.
    \end{align*}
    Hence
    \begin{align*}
        \partial_\xi N(\DD) 
        &= -2\xi\left(\frac{s+1}{2}\right) \begin{pmatrix}
             (\xi^2+\Dd^-\Dd^+)^{-\frac{s+3}{2}} & 0 \\ 0 & (\xi^2+\Dd^+\Dd^-)^{-\frac{s+3}{2}}
        \end{pmatrix}.      
    \end{align*}
    Since  $\tnablax N(\Dt) = \begin{pmatrix}
        0 & \dot{\Dd}^- \\ \dot{\Dd}^+ & 0
    \end{pmatrix}$, 
    we have that

    \begin{multline*}
         \bTr(\partial_\xi N(\DD) \tnablax N(\Dt) )  \\= -2\xi \left(\frac{s+1}{2}\right) \bTr \begin{pmatrix}
              0 & (\xi^2+\Dd^-\Dd^+)^{-\frac{s+3}{2}}\dot{\Dd}^- \\ (\xi^2+\Dd^+\Dd^-)^{-\frac{s+3}{2}}\dot{\Dd}^+ & 0
         \end{pmatrix}
         = 0.
    \end{multline*}
\end{proof}

Plugging in the results of Lemmas \ref{lem_btrnormal} and \ref{lem_btrderivs} into \eqref{eqtrace} yields
\begin{equation}
    \lim_{t\to 0}\bTr((\Dt^2)^{-\frac{s+1}{2}}\Dt) = \int_{T^*\S^1} \bTr (q_1(\DD)N(\Dt)).
    \label{eq_bTrq1N}
\end{equation}
From here we have that
\begin{align}
    \bTr(q_1(\DD) N(\Dt))
    &= \frac{1}{\Gamma(\frac{s+1}{2})}\int_0^\infty \tau^\frac{s+1}{2} \bTr( q_1(e^{-\tau\Dt^2})N(\Dt))\,\frac{d\tau}{\tau}.
    \label{eq_q1N}
\end{align}
Observe that
\begin{equation*}
    \Dt^2 = \begin{pmatrix}
        (-it\tnablax)^2+\Dd^-\Dd^+ & 0 \\ 0 & (it \tnablax)^2+ \Dd^+\Dd^-
    \end{pmatrix} + t\begin{pmatrix}
        0 & -i\dot{\Dd}^- \\ i \dot{\Dd}^+ & 0
    \end{pmatrix}.
\end{equation*}
If we define $\alpha$ and $\beta$ by the matrices above, we can apply Duhamel's formula twice with $\Dt^2 = \alpha+t\beta$ to get
\begin{align*}
    q_1(e^{-\tau\Dt^2})
    &= -\int_0^\tau e^{-(\tau-r)N(\alpha)} N(\beta) e^{-rN(\alpha)}\, dr.
\end{align*}
Plugging this into \eqref{eq_q1N} yields
\begin{equation*}
    -\frac{1}{\Gamma(\frac{s+1}{2})}\int_0^\infty \int_0^\tau \tau^\frac{s+1}{2} \bTr( e^{-(\tau-r)N(\alpha)} N(\beta) e^{-rN(\alpha)}N(\Dt))\,dr \frac{d\tau}{\tau}.
\end{equation*}
A simple computation using  the facts that $e^{-\tau\Dd^-\Dd^+}\Dd^- = \Dd^-e^{-\tau\Dd^+\Dd^-}$ and $e^{-\tau\Dd^+\Dd^-}\Dd^+ = \Dd^+e^{-\tau\Dd^-\Dd^+}$ shows that $e^{-rN(\alpha)}N(\Dt) = N(\Dt) e^{-rN(\alpha)}$. Thus the above expression becomes
\begin{equation}
    -\frac{1}{\Gamma(\frac{s+1}{2})} \int_0 ^\infty \int_0^\tau  \tau^\frac{s+1}{2}\bTr(e^{-\tau N(\alpha)}N(\beta)N(\Dt))\,dr \frac{d\t}{\t} + R_1,
    \label{eq_bTrR1}
\end{equation}
where $R_1$ is defined in the statement of Lemma \ref{lem_R1}. This remainder term is due to the fact that the $b$-trace of a commutator does \textit{not} necessarily vanish.

For now, we will focus on the first term of \eqref{eq_bTrR1}. Since the integrand is independent of $r$, the term becomes
\begin{align*}
    & -\frac{1}{\Gamma(\frac{s+1}{2})} \int_0^\infty \t^\frac{s+3}{2} \bTr(e^{-\t N(\alpha)} N(\beta)N(\Dt))\,\frac{d\t}{\t}.
\end{align*}
Bringing the $b$-trace back to the front and using $\frac{s+1}{2}\Gamma(\frac{s+1}{2}) = \Gamma(\frac{s+3}{2})$, we have
\begin{equation}
    -\frac{s+1}{2}\bTr(N(\alpha)^{-\frac{s+3}{2}}N(\beta)N(\Dt)).
    \label{eq_N(A)N(B)N(Dt)}
\end{equation}
Observe that
\begin{equation*}
    N(\alpha)^{-\frac{s+3}{2}}N(\beta)N(\Dt) =
    \begin{pmatrix}
        -i(\xi^2+\Dd^-\Dd^+)^{-\frac{s+3}{2}}\dot{\Dd}^- \Dd^+ & * \\ * &  i(\xi^2+ \Dd^+\Dd^-)^{-\frac{s+3}{2}} \dot{\Dd^+} \Dd^-
    \end{pmatrix}.
\end{equation*}
It follows that \eqref{eq_N(A)N(B)N(Dt)} becomes
\begin{equation}
     \frac{s+1}{2i}(\bTr((\xi^2+ \Dd^+\Dd^-)^{-\frac{s+3}{2}} \dot{\Dd^+} \Dd^-) - \bTr((\xi^2+\Dd^-\Dd^+)^{-\frac{s+3}{2}}\dot{\Dd}^- \Dd^+)).
    \label{eq_J1J2}
\end{equation}
Focusing on the first term of \eqref{eq_J1J2}, observe that
\begin{align}
    \bTr((\xi^2+\Dd^+\Dd^-)^{-\frac{s+3}{2}}\dot{\Dd}^+\Dd^-) &= \bTr(\Dd^-\dot{\Dd}^+(\xi^2+\Dd^-\Dd^+)^{-\frac{s+3}{2}}) + R_2 +R_3,
    \label{eq_integrand1}
\end{align}
where $R_2$ and $R_3$ are defined in the statement of Lemma \ref{lem_R2R3}.
Focusing on the second term of \eqref{eq_J1J2}, observe that
\begin{multline}
     - \bTr((\xi^2+\Dd^-\Dd^+)^{-\frac{s+3}{2}}\dot{\Dd}^- \Dd^+) 
     \\ =\bTr((\xi^2+\Dd^-\Dd^+)^{-\frac{s+3}{2}}\Dd^-\dot{\Dd}^+) - \bTr((\xi^2+\Dd^-\Dd^+)^{-\frac{s+3}{2}}[\tnablax,\Dd^-\Dd^+]).
     \label{eq_integrand2}
\end{multline}
We will now take \eqref{eq_integrand1} and \eqref{eq_integrand2} and plug them into \eqref{eq_J1J2}, and then that result will be plugged into the first term of \eqref{eq_bTrR1}. Finally, the entirety of \eqref{eq_bTrR1} will be inserted into the integrand of \eqref{eq_bTrq1N}. To simplify this massive expression, we will use the following across each term:
\begin{align*}
    \frac{s+1}{2}\int_{-\infty}^\infty(\xi^2+\Dd^-\Dd^+)^{-\frac{s+3}{2}}\,\dbar\xi & =  \frac{1}{2\sqrt{\pi}} \frac{\Gamma(\frac{s}{2}+1)}{\Gamma(\frac{s+1}{2})}(\Dd^-\Dd^+)^{-\frac{s}{2}-1}.
\end{align*}
This leaves us with
\begin{multline*}
    \lim_{t\to 0} \bTr(\DDD) = \frac{1}{2i\sqrt{\pi}}\frac{\Gamma(\frac{s}{2}+1)}{\Gamma(\frac{s+1}{2})}\int_{\S^1}\bigg( \bTr(\Dd^-\dot{\Dd}^+(\Dd^-\Dd^+)^{-\frac{s}{2}-1})\\+\bTr((\Dd^-\Dd^+)^{-\frac{s}{2}-1}\Dd^-\dot{\Dd}^+) -\bTr((\Dd^-\Dd^+)^{-\frac{s}{2}-1}[\tnablax,\Dd^-\Dd^+])\bigg) +R,
\end{multline*}
where $R= \int_{T^*\S^1}(R_1+\frac{s+1}{2i}R_2+\frac{s+1}{2i}R_3)\,\dbar\xi$. From here we can use Proposition \ref{lem_bzeta} on the last term of the integrand to get
\begin{multline*}
    \frac{1}{2i\sqrt{\pi}}\frac{\Gamma(\frac{s}{2}+1)}{\Gamma(\frac{s+1}{2})}\int_{\S^1}\bigg( \bTr(\Dd^-\dot{\Dd}^+(\Dd^-\Dd^+)^{-\frac{s}{2}-1})\\+\bTr((\Dd^-\Dd^+)^{-\frac{s}{2}-1}\Dd^-\dot{\Dd}^+) +\frac{1}{\Gamma(\frac{s}{2}+1)}\partial_x\, ^b\overline\zeta(s,\Dd^-\Dd^+)\bigg) +R.
\end{multline*}
The $^b\overline{\zeta}$ term will vanish under the integral by the Fundamental Theorem of Calculus.
Recall that $^b\overline{\eta}(s,\Dt) = \frac{\Gamma(\frac{s+1}{2})}{\sqrt{\pi}}\bTr((\Dt^2)^{-\frac{s+1}{2}}\Dt) $.
Therefore,
\begin{equation*}
    \lim_{t\to 0}i\pi\,^b\overline{\eta}(s,\Dt) = \int_{\S^1}\,^b\omega(s) + i\sqrt{\pi}\Gamma\left(\frac{s+1}{2}\right) R.
\end{equation*}
We will now simplify the remainder terms. Putting Lemma \ref{lem_R1} and Lemma \ref{lem_R2R3} together yields 
\begin{align*}
    R &= \frac{1}{8i\pi}\int_{\S^1}\Tr \left(  (\Dd^2)^{-\frac{s+1}{2}}(G^*\dot{G}\Dd - 3G^* \dot{P} +\dot{P}^*G - \dot{G}^* P)   \right).
\end{align*}
Using the algebraic relations of the operators, the above expression can be rewritten as
\begin{equation*}
    R = -\frac{1}{4i\pi}\int_{\S^1}\Tr \left(  (\Dd^2)^{-\frac{s+1}{2}}(G^*\dot{G}\Dd +\dot{\Dd})  \right).
\end{equation*}
Applying Proposition \ref{lem_bEta} gives us that
\begin{align*}
    R &= -\frac{1}{4i\pi}\int_{\S^1}\Tr \left(  (\Dd^2)^{-\frac{s+1}{2}}(G^*\dot{G}\Dd  ) \right) +\frac{1}{4i s\sqrt{\pi}\Gamma(\frac{s+1}{2})}\int_{\S^1}\D_x\, \overline{\eta}(s,\D M).
\end{align*}
Again, the $\overline{\eta}$ term will vanish under the integral by the Fundamental Theorem of Calculus. Therefore
\begin{equation*}
    i\sqrt{\pi}\Gamma\left(\frac{s+1}{2}\right)R =  -\frac{\Gamma(\frac{s+1}{2})}{4\sqrt{\pi}} \int_{\S^1}\Tr ((\Dd^2)^{-\frac{s+1}{2}}G^*\dot{G}\Dd),
\end{equation*}
and the proof of Theorem \ref{thm_GlobalAnomaly} is complete.

\subsection{Corollaries}


Parallel transport in $\mathcal{L}$ along $B$ is defined as
\[
\mathrm{PT}(\mathcal{L}) = e^{- \int_B \b \omega}.
\]
Since $B = \S^1$ represents a closed curve, the holonomy of $\mathcal{L}$ is defined as $\mathrm{hol} (\mathcal{L}) := \mathrm{PT}(\mathcal{L})$. For the sake of simplicity we chose to focus our attention on the case where $B=\S^1$. However, if we are interested in the parallel transport in $\mathcal{L}$ along a curve in the base, we can take $B$ to be diffeomorphic to $\overline{\R}$, the exponential compactification of $\R$ to a closed interval. The curve may be closed if the end points of $\overline{\R}$ represent the same point in the base, in which case parallel transport measures the holonomy and the problem reduces to the case with $B=\S$. Very few adjustments are needed to adapt the contents of this paper to this setting, and the extra terms in the subsequent theorem arise from the Fundamental Theorem of Calculus. We now state the analogous main result:
\begin{theorem}
The normalized $b$-eta function $\beb(s,\Dd_t)$ extends smoothly to
$t = 0$ and in the convergence of meromorphic functions of $s \in \C$,
\begin{multline*}
\notag \lim_{t \to 0} i \pi\, \beb(s,\Dd_t) - \int_B \b \omega (s) = \\
 \frac12
\Big\{{}^b \overline{\zeta}(s,\Dd^-_{\infty} \Dd^+_{\infty}) - {}^b \overline{\zeta}(s,\Dd^-_{-\infty} \Dd^+_{-\infty}) \Big\} 
  +\frac{1}{4 s} \Big\{ \overline{\eta}(s,\D M_\infty) - \overline{\eta}(s,\D
M_{-\infty}) \Big\} \\
 -\frac{\Gamma \left( \frac{s + 1}{2} \right)}{4 \sqrt{\pi}}
\int_{B} \Tr \Big( (\Dd^2)^{- \frac{s
+ 1}{2}} \, \Dd \, G^{*} \, \dot{G}  \Big) .
\end{multline*}
\end{theorem}
More details on this result will be presented in \cite{Corners}. The subscripts $\infty$ and $-\infty$ denote the right and left end points, respectively, in $\overline{\R}$. Thus, the operators $\Dd^\pm_\infty$ (respectively $\Dd^\pm_{-\infty}$) are the restrictions of $\Dd^\pm$ to the right (respectively left) end point of $\overline{\R}$ corresponding to $+\infty$ (respectively $-\infty$) in $\R$. The functions ${}^b \overline{\zeta}(s,\Dd^\mp_{\infty} \Dd^\pm_{\infty})$ denote normalized $b$-zeta functions. The terms $\overline{\eta}(s,\D M_{\pm\infty})$ represent the eta function restricted to the right and left end points of $\overline{\R}$. When the curve is closed so that $\Dd^\pm_\infty = \Dd^\pm_{-\infty}$, the first and second terms on the right cancel and we are left with the original result.

By taking the regular of $s = 0$ in Theorem \ref{thm-main} we recover, for free, a generalization of the Bismut-Freed theorem. The second theorem below is for the case where $\dot{G}=0$.

\begin{theorem}  $\be(\Dd_t) := \beb(0,\Dd_t)$, the $b$-eta invariant, extends smoothly to $t = 0$
and
\begin{multline*}
\lim_{t \to 0} i \pi\, \be(\Dd_t) + \log \mathrm{PT}(\mathcal{L})  = \\ -\frac{1}{4} \,
\int_{\S^1} \Tr_{[0]} \Big( (\Dd^2)^{-
\frac{1}{2}} \, \Dd \, G^{-1}
\, \dot{G}  \Big)
 + \frac{1}{8} (\gamma + \log 4) \, 
\int_{\S^1} \Tr_{[-1]} \Big( (\Dd^2)^{- \frac{1}{2}} \, \Dd\,  G^{-1}
\, \dot{G} \Big),
\end{multline*}
where $\Tr_{[k]} \Big( (\Dd^2)^{-
\frac{1}{2}} \, \Dd \, G^{-1}
\, \dot{G}  \Big)$ denotes the coefficient of $s^{k}$ in the Laurent expansion of $\Tr \Big( (\Dd^2)^{- \frac{s + 1}{2}} \, \Dd\, G^{-1} \, \dot{G}
\Big)$ at $s = 0$.
\end{theorem}

\begin{theorem} \label{thm-main3} If the unitary map $G$ is parallel at the boundary, then as meromorphic functions of $s \in \C$,
\begin{equation*}
\notag \lim_{t \to 0} i \pi\, \beb(s,\Dd_t) = \int_{\S^1} \b \omega (s).
\end{equation*}
\end{theorem}